\documentclass[a4paper]{article}
\usepackage{amsmath,amsthm,amssymb,amscd,ascmac, amsfonts}
\usepackage{mathrsfs}
\usepackage{stmaryrd}
\usepackage{braket}
\usepackage{color}
\usepackage{accents}
\usepackage{url}
\usepackage[english]{babel}
\usepackage[margin=20truemm]{geometry}
\usepackage[title,titletoc]{appendix}
\usepackage[pdftex]{graphicx}
\allowdisplaybreaks[4]

\newcommand{\Zz}{\mathbb{Z}}
\newcommand{\Cz}{\mathbb{C}}
\newcommand{\Rz}{\mathbb{R}}

\newcommand{\hyper}[4]{\left( \begin{matrix} #1 \\ #2 \end{matrix};#3;#4\right)}

\numberwithin{equation}{section}
\theoremstyle{plain}
\newtheorem{thm}{Theorem}
\newtheorem{defi}{Definition}
\newtheorem{prop}{Proposition}
\newtheorem{lem}{Lemma}
\newtheorem{coro}{Corollary}

\newtheorem{rmk}{Remark}

\newtheorem{exg}{Example}

\begin{document}
\title{A degeneration of the generalized Zwegers' $\mu$-function according to the Ramanujan difference equation}
\author{G. Shibukawa and S. Tsuchimi\\
(Kitami Institute of Technology and Kindai University, Japan)}
\date{}

\maketitle

\begin{abstract}
In this paper, we introduce the little $\mu$-function, which is obtained as a degenerate limit of the generalized $\mu$-function. 
We derive the little $\mu$-function as the image of the $q$-Borel summation of a divergent solution to the Ramanujan equation one of the most degenerate second order linear $q$-difference equations of Laplace type excluding those of constant coefficients. 
Moreover, we present several formulas, such as symmetries and connection formulas for the little $\mu$-function, similar to those for the generalized $\mu$-function. 
Furthermore, we establish contiguous relations related to the $q,t$-Fibonacci sequences and evaluate Wronski determinants involving the quadratic relation of the Rogers-Ramanujan continued fraction. 
\end{abstract}

\section{Introduction}

Throughout this paper, let $i:=\sqrt{-1}$ be the imaginary unit, $\tau\in\Cz$ be a complex number with ${\rm Im}(\tau)>0$, $q:=e^{2\pi i\tau}$ and $a:=q^\alpha\ (\alpha\in\Cz)$. 
We define the $q$-shifted factorials and the Jacobi theta function as follows: 
\begin{align*}
(x)_\infty
  &=
  (x;q)_\infty
  :=
  \prod_{n=0}^\infty(1-x q^n), \quad  
  (x)_\alpha
  =
  (x;q)_\alpha
  :=
  \frac{(x)_\infty}{(a x)_\infty}\quad (\alpha\in\Cz), 
  \\
  (a_1,\ldots,a_r)_\beta
  &=
  (a_1,\ldots,a_r;q)_\beta
  :=
  \prod_{j=1}^r(a_j)_\beta
  \quad (\beta\in\Cz\cup\{\infty\}), 
  \\
\theta(x;q)&=\theta(x):=(x,q/x)_\infty=\frac{1}{(q)_\infty}\sum_{n\in\Zz}(-x)^nq^\frac{n(n-1)}{2}. 
\end{align*}
For appropriate complex numbers $a_1,\ldots,a_r,b_1,\ldots,b_s,x$, we denote the $q$-hypergeometric series as follows: 
\begin{align*}
{}_r\phi_s\hyper{a_1,\ldots,a_r}{b_1,\ldots,b_s}{q}{x}
  &:=
  \sum_{n=0}^\infty
  \frac{(a_1,\ldots,a_r)_n}{(b_1,\ldots,b_s,q)_n}
    \left((-1)^nq^\frac{n(n-1)}{2}\right)^{s-r+1} x^n
    , \\
{}_r\psi_s\hyper{a_1,\ldots,a_r}{b_1,\ldots,b_s}{q}{x}
  &:=
  \sum_{n\in\Zz}
  \frac{(a_1,\ldots,a_r)_n}{(b_1,\ldots,b_s)_n}
    \left((-1)^nq^\frac{n(n-1)}{2}\right)^{s-r} x^n. 
\end{align*}

The authors introduced the generalized $\mu$-function as follows \cite{ST1}:
\begin{align}
\label{eq: defi of gmu}
\widehat{\mu}(x,y;a)
   &:=
   iq^{-\frac{1}{8}}
   \frac{(xy)^{\frac{\alpha}{2} }(ax)_{\infty }}{(x,q)_{\infty }\theta (y)}
   {}_1\psi_2\left(\begin{matrix} x \\ 0,ax \end{matrix};q,\frac{q}{y}\right)\quad (x,y\not\in q^{\Zz}:=\{q^n\in\Cz; n\in\Zz\}),
\end{align}
but, here we use the definition given in \cite{ST2}. 
The original Zwegers $\mu$-function \cite{Zw1} is a special case of our generalized $\mu$-function : 
\begin{align*}
\widehat{\mu}(x,y;q)=\mu(x,y):=iq^{-\frac{1}{8}}\frac{\sqrt{xy}}{(q)_\infty\theta(y)}\sum_{n\in\Zz}\frac{(-1)^ny^nq^\frac{n(n+1)}{2}}{1-x q^n}, 
\end{align*}
which includes a family of $q$-series called Ramanujan's mock theta functions and satisfies many formulas such as $q$-difference relations and mock modular properties \cite{BFOR}. 

In recent years, there have been several studies of these $\mu$-functions from the viewpoint of $q$-Borel summable methods (for example, see \cite{GW}, \cite{ST1} and \cite{ST2}). 
$q$-Borel summable methods are one of the techniques for constructing convergent solutions from divergent solutions of $q$-difference equations. 
In this paper, for the formal series $g(x)=\sum_{n=0}^\infty a_n x^n$, we define the $q$-Borel transformation $\mathcal{B}_q$ and the $q$-Laplace transformation $\mathcal{L}_q$ as follows: 
\begin{align}
\label{eq: q-Borel Laplace transformation}
\mathcal{B}_q(g)(\xi):=\sum_{n=0}^\infty a_nq^\frac{n(n-1)}{2}\xi^n,\quad \mathcal{L}_q(g)(x;\lambda):=\sum_{n\in\Zz}\frac{g(\lambda q^n)}{\theta_q(\lambda q^n/x)}=\frac{1}{1-q}\int_0^{\lambda \infty}\frac{g(t)}{\theta_q(t/x)}\ d_qt, 
\end{align}
where the integral in the last term of the $q$-Laplace transformation is called the Jackson integral and is defined as the following sum: 
\begin{align*}
\int_0^{\lambda\infty} f(t)\ d_qt:=(1-q)\sum_{n\in\Zz}^\infty f(\lambda q^n). 
\end{align*}
Note that the parameter $\lambda$ appearing in the $q$-Laplace transformation is regarded as the argument of the integral path of the Jackson integral. 

For $\varphi_n(x):=x^n\ (n\in\Zz)$, a simple calculation shows that
\begin{align}
\label{eq: mono summation}
\mathcal{L}_q\circ\mathcal{B}_q(\varphi_n)(x,\lambda)
  &=x^n. 
\end{align}
Namely, from \eqref{eq: mono summation}, if we apply the composite of the $q$-Borel transformation and the $q$-Laplace transformation $\mathcal{L}_q\circ\mathcal{B}_q$ to a convergent series, the image is equal to the original convergent series. 
If we obtain a convergent series when we apply $\mathcal{L}_q\circ\mathcal{B}_q$ to a divergent solution $g(x)$ of a $q$-difference equation, the function $\mathcal{L}_q\circ\mathcal{B}_q(g)(x,\lambda)$ is a convergent solution of the original $q$-difference equation (see \cite{RSZ}). 

As an example of a $q$-difference equation in which divergent solutions arise, we present the following equation that is called the $q$-Hermite-Weber equation: 
\begin{align}
\label{eq: q-HW equation}
\left[T_x^2-(1-a x)T_x-x\right]f(x)=0,\quad T_xf(x):=f(q x). 
\end{align}
A convergent solution of the $q$-Hermite-Weber equation \eqref{eq: q-HW equation} around $x=0$ is given by 
\begin{align*}
\widetilde{g}_0(x;a):=\frac{1}{\theta(x)}{}_1\phi_1\hyper{q/a}{0}{q}{ax}=\frac{(ax)_\infty}{\theta(x)}{}_0\phi_1\hyper{-}{ax}{q}{xq}, 
\end{align*}
and a divergent solution is
\begin{align*}
\widetilde{g}_1(x;a):=\sum_{n=0}^\infty \frac{(a)_n}{(q)_n}q^{-\frac{n(n+1)}{2}}(-x)^n={}_2\phi_0\hyper{a,0}{-}{q}{\frac{x}{q}}. 
\end{align*}
The relation between the generalized $\mu$-function $\widehat{\mu}$ and the divergent solution $\widetilde{g}_1(x;a)$ is 
\begin{align*}
\widehat{\mu}(x,y;a;q)=i(xy)^\frac{\alpha}{2}q^{-\frac{1}{8}}\mathcal{L}_q\circ\mathcal{B}_q\left(\widetilde{g}_1\right)(xy,-x). 
\end{align*}

Furthermore, we gave some formulas of the generalized $\mu$-function such as the $q$-difference equation, symmetries, expression of ${}_0\psi_2$ and the very-well-poised bilateral $q$-hypergeometric series, corresponding to the convergent solution and the connection formulas \cite[Theorem 1.3]{ST1}, \cite[$(1.37)$--$(1.42)$]{ST2}: 
\begin{align}
\label{eq: gmu formula 1}
\widehat{\mu}(x q^2,y;a)
  &=
  \sqrt{a}(1-axy)\widehat{\mu}(x q,y;a)+a x y \widehat{\mu}(x,y;a), \\
\label{eq: gmu formula 2}
\widehat{\mu}(x,y;a)
  &=
  \widehat{\mu}(x/q,yq;a)
    =
    \widehat{\mu}(y,x;a), \\
\label{eq: gmu formula 3}
\widehat{\mu}(x,y;a)
  &=
  iq^{-\frac{1}{8}}\frac{(xy)^\frac{\alpha}{2}(a,ax)_\infty}{(y,q)_\infty\theta(x)}{}_0\psi_2\hyper{-}{q/y,ax}{q}{\frac{xq}{y}}, \\
\label{eq: gmu formula 4}
\widehat{\mu}(x,y;a)
  &=
  iq^{-\frac{1}{8}}\frac{(xy)^\frac{\alpha}{2}(ax,y/a)_\infty}{(x,y,q)_\infty\theta(x/y)}\sum_{n\in\Zz}\left(1-\frac{x}{y}q^{2n}\right)\frac{(x,aq/y)_n}{(x/a,q/y)_n}q^{n(2n-1)}\left(\frac{x^2}{ay^2}\right)^n, \\
\label{eq: gmu formula 5}
\lim_{y\to1}\theta(y)\widehat{\mu}(x,y;a)
  &=
  iq^{-\frac{1}{8}}\frac{x^\frac{\alpha}{2}(a,ax)_\infty}{\theta(x)}{}_0\phi_1\hyper{-}{ax}{q}{xq}, \\
\widehat{\mu}(x,y;a)
  &=
  \frac{\theta(x/c_1)\theta(y c_1)\theta(c_2)\theta(c_2y/x)}{\theta(x)\theta(y)\theta(c_1c_2y/x\theta(c_2/c_1))}\widehat{\mu}(x/c_1,yc_1;a) \nonumber \\
\label{eq: gmu formula 6}
    &\qquad +\frac{\theta(x/c_2)\theta(y c_2)\theta(c_1)\theta(c_1y/x)}{\theta(x)\theta(y)\theta(c_1c_2y/x)\theta(c_1/c_2)}\widehat{\mu}(x/c_2,yc_2;a), \\
\label{eq: gmu formula 7}
\widehat{\mu}(x,y;a)
  &=
  \widehat{\mu}(x/c,yc;a)
    +\frac{iq^{-\frac{1}{8}}(xy)^\frac{\alpha}{2}\theta(c)\theta(cy/x)}{\theta(x)\theta(y)\theta(c/x)\theta(cy)}(a,axy)_\infty{}_0\phi_1\hyper{}{axy}{q}{xq}. 
\end{align}

The $q$-Hermite-Weber equation \eqref{eq: q-HW equation} is an example of second order $q$-difference equations of Laplace type
\begin{align}
\label{eq: Laplace}
[(a_0+b_0x)T_x^2+(a_1+b_1x)T_x+(a_2+b_2x)]f(x)=0, 
\end{align}
and is given by taking a few degenerate limits of the following $q$-difference equation satisfied by the Heine's $q$-hypergeometric series ${}_2\phi_1$ which is a master class of Laplace type (see \cite[Figure 2]{O}): 
\begin{align}
\label{eq: Heine eq}
[(c-abqx)T_x^2-(c+q-(a+b)qx)T_x+q(1-x)]f(x)=0. 
\end{align}
The fundamental solutions around $x=0$ of \eqref{eq: Heine eq} are as follows:  
\begin{align}
\label{eq: Heine sol}
\frac{\theta(x)}{\theta(xq/c)}{}_2\phi_1\hyper{aq/c,bq/c}{q^2/c}{q}{x}=\frac{(q/x,abx/c)_\infty}{\theta(x q/c)}{}_2\phi_1\hyper{q/a,q/b}{q^2/c}{q}{\frac{abx}{c}},\quad {}_2\phi_1\hyper{a,b}{c}{q}{x}, 
\end{align}
and formal solutions of the $q$-Hermite-Weber equation \eqref{eq: q-HW equation} are also obtained by the following degenerate limits of \eqref{eq: Heine sol}: 
\begin{align}
\label{eq: convergent solution of q-HW}
\widetilde{g}_0(x;a)
  &=
  \lim_{\substack{b\to0, \\c\to0}}\frac{(cq^2/x,abx/q)_\infty}{\theta(x)}{}_2\phi_1\hyper{q/a,q/b}{c q^2}{q}{\frac{abx}{q}}
  , \\
\label{eq: divergent solution of q-HW}
 \widetilde{g}_1(x;a)
  &=
  \lim_{\substack{b\to0, \\c\to0}}{}_2\phi_1\hyper{a,b}{1/c}{q}{\frac{x}{cq}}. 
\end{align}

By substituting an appropriate scaling transformation and specialization $a=0$ of the $q$-Hermite-Weber equation \eqref{eq: q-HW equation}, we derive the following $q$-difference equation: 
\begin{align}
\label{eq: Ramanujan equation}
\left[T_{x}^{2} - T_{x} - qxy \right]f(x)=0. 
\end{align}
The above equation \eqref{eq: Ramanujan equation} is one of the most degenerate examples of Laplace type linear $q$-difference equations excluding those of constant coefficients. 
Involving \eqref{eq: Ramanujan equation}, there are six such most degenerate equations from \eqref{eq: Laplace} and they are transformed into one another by suitable gauge transformations (see Appendix \ref{app: A}). 
Hence, in this paper, we only focus on \eqref{eq: Ramanujan equation}, and we call the ``Ramanujan equation''. 

A convergent solution and a divergent solution of the Ramanujan equation \eqref{eq: Ramanujan equation} around $x=0$ are derived from the following degenerate limits of \eqref{eq: convergent solution of q-HW} and \eqref{eq: divergent solution of q-HW}: 
\begin{align*}
\widetilde{f}_0(xy)&:=\frac{1}{\theta(xy q)}{}_0\phi_{1}\left(\begin{matrix} - \\ 0 \end{matrix};q,xy q^2\right)=\lim_{a\to0}\widetilde{g}_0(xyq;a), \\
\widetilde{f}_1(x)&:={}_2\phi_0\hyper{0,0}{-}{q}{xy}=\lim_{a\to0}\widetilde{g}_1(xyq;a). 
\end{align*}
We remark that $\theta(xq)\widetilde{f}_0(x/q)$ is sometimes called the Ramanujan entire function. 
The special cases of $\theta(xq)\widetilde{f}_0(x/q)$ coincide with the series of the Rogers-Ramanujan identities:
\begin{align*}
G(q)=\theta(q)\widetilde{f}_0(1/q),\quad H(q)=\theta\left(q^2\right)\widetilde{f}_0(1), 
\end{align*} 
where
\begin{align*}
G(q)&:=\sum_{n=0}^\infty\frac{q^{n^2}}{(q;q)_n}=\frac{1}{\left(q,q^4;q^5\right)_\infty}, \quad H(q):=\sum_{n=0}^\infty\frac{q^{n^2+n}}{(q;q)_n}=\frac{1}{\left(q^2,q^3;q^5\right)_\infty}. 
\end{align*}

Also, for $n\in\Zz$, the convergent solution $\widetilde{f}_0$ is expressed by a linear combination of the $q,t$-Fibonacci sequences: 
\begin{align}
\widetilde{f}_0\left(xq^{n-1}\right)&=\widetilde{f}_0(xq^{-1})T_{n-1}(x,q)+\widetilde{f}_0(x)S_{n}(x,q)
, \\
\label{eq: Garret, Ismail formula}
\widetilde{f}_0(q^{n-1})&=\frac{G(q)}{\theta(-q)}T_{n-1}(q)+\frac{H(q)}{\theta(-q^2)}S_n(q). 
\end{align}
The formula equals to \cite[(3.2)]{GIS}. 
The sequences $S_n(t,q)$ and $T_{n-1}(t,q)$ are the unique solutions that satisfy the following recursion
\begin{align}
\label{eq: q-Fibonacci}
F_n(t,q)=F_{n-1}(t,q)+tq^{n-2}F_{n-2}(t,q)
\end{align}
with the initial values 
\begin{align*}
\begin{cases}
F_0(t,q)=S_0(t,q)=0\\
F_1(t,q)=S_1(t,q)=1
\end{cases}, \quad 
\begin{cases}
F_0(t,q)=T_{-1}(t,q)=1\\
F_1(t,q)=T_0(t,q)=0
\end{cases}, 
\end{align*}
respectively \cite{A}, \cite{C}. 
Also, when $t=1$, the sequences
\begin{align}
\label{eq: Schur Fibo}
S_n(q):=S_n(1,q),\quad T_n(q):=T_n(1,q)
\end{align}
are the Schur's $q$-Fibonacci sequences \cite{S}.
Note that the recursion of the $q,t$-Fibonacci sequences \eqref{eq: q-Fibonacci} is obtained by specializing the Ramanujan equation \eqref{eq: Ramanujan equation} with 
\begin{align}
\label{eq: corresponding between q-Fibonacci and Raman ujan eq}
x\mapsto tq^{n-1}/y,\quad f\left(tq^{n-1}/y,y\right)=F_n(t,q). 
\end{align} 

On the other hand, from the $q$-Borel summable method of the divergent solution $\widetilde{f}_1$, we obtain the following function which we call the little $\mu$-function. 
\begin{defi}
For $x, y\not\in q^{\mathbb{Z}} $, we define the following series: 
\begin{align*}
l\widehat{\mu}(x,y)
   &:=
   iq^{-\frac{1}{8}}\frac{1}{(x,q)_{\infty }\theta(qy)}
   {}_1\psi_2\left(\begin{matrix} x \\ 0,0 \end{matrix};q,\frac{1}{y}\right).
\end{align*}
\end{defi}

In this paper, first, we show that the little $\mu$-function coincides with the image of $\mathcal{L}_q\circ\mathcal{B}_q$ of the divergent solution $\widetilde{f}_1$, and is obtained by a degenerate limit of the generalized $\mu$-function. 
\begin{thm}
\label{thm: lmu and gmu}
We have the following equations: 
\begin{align}
\label{eq: lmu and divergent series}
  l\widehat{\mu}(x,y)
   &=iq^{-\frac{1}{8}}\mathcal{L}_q\circ\mathcal{B}_q\left(\widetilde{f}_1\right)(x,-x/y)
, \\
\label{eq: lmu and gmu}
l\widehat{\mu}(x,y)
      &=\lim_{a \to 0}
      (xy)^{-\frac{\alpha }{2}}a^{-\frac{1}{2}}
      \widehat{\mu}(q x,y;a)
      =\lim_{a \to 0}
      (xy)^{-\frac{\alpha }{2}}a^{-\frac{1}{2}}
      \widehat{\mu}(x,qy;a). 
\end{align}
\end{thm}
Moreover, we present a $q$-difference equation, symmetries, expressions by ${}_0\psi_2$ and very-well-poised bilateral $q$-hypergeometric series, corresponding to the convergent solution, connection formulas and contiguous relations of the little $\mu$-function. 
\begin{thm}
\label{thm: lmu property}
We obtain
\begin{align}
\label{eq: lmu formula 1}
l\widehat{\mu}(x,y)
   &=
   l\widehat{\mu}(q x,y)-x y l\widehat{\mu}(x/q,y), \\
\label{eq: lmu formula 2}
l\widehat{\mu}(x,y)
   &=
   l\widehat{\mu}(x/q,q y)
   =
   l\widehat{\mu}(y,x), \\
\label{eq: lmu formula 3}
l\widehat{\mu}(x,y)
  &=
  \frac{iq^{-\frac{1}{8}}}{(qy,q)_\infty\theta(x)}{}_0\psi_2\hyper{-}{1/y,0}{q}{\frac{x}{y}}, \\
\label{eq: lmu formula 3.5}
l\widehat{\mu}(x,y)
  &=
  \frac{iq^{-\frac{1}{8}}}{(x,y q)_\infty\theta(x/y q)}\sum_{n\in\Zz}\left(1-\frac{x}{y}q^{2n-1}\right)\frac{(x)_n}{(1/y)_n}q^{\frac{5}{2}n(n-1)}\left(-\frac{x^2}{y^3 q}\right)^n, \\
\label{eq: lmu formula 4}
\lim_{y\to 1}\theta(y)l\widehat{\mu}(x,y)
  & =
  \frac{iq^{-\frac{1}{8}}}{\theta(qx)}{}_0\phi_1\hyper{-}{0}{q}{x q^2}, \\
\label{eq: lmu formula 5}
l\widehat{\mu}(x,y)
   &=
   \frac{\theta(x/c_1)\theta(c_1 y)\theta(c_2y/x)\theta(c_2)}{\theta(x)\theta(y)\theta(c_1c_2y/x)\theta(c_2/c_1)}
      l\widehat{\mu}(x/{c}_{1},y{c}_{1})
   +
   \frac{\theta(x/{c}_{2})\theta({c}_{2}y)\theta ({c}_{1}y/x)\theta ({c}_{1})}{\theta (x)\theta (y)\theta ({c}_{2}{c}_{1}y/x)\theta ({c}_{1}/{c}_{2})}
   l\widehat{\mu}(x/{c}_{2},y{c}_{2})\\
\label{eq: lmu formula 6}
   &=
   l\widehat{\mu}(x/c,y c)
     -\frac{iq^{-\frac{1}{8}}\theta(c)\theta (y c/x)}{\theta (x q)\theta (y q)\theta (c/x)\theta (y c)} 
       {}_0\phi_{1}\left(\begin{matrix} - \\ 0 \end{matrix};q,x y q^2\right).
\end{align}
For $n\in\Zz$, we have
\begin{align}
\label{eq: lmu formula 7}
l\widehat{\mu}(xq^{n-1},y)&=l\widehat{\mu}(x q^{-1},y)T_{n-1}(xy,q)+l\widehat{\mu}(x,y)S_{n}(xy,q). 
\end{align}
\end{thm}

The above formulas of Theorem \ref{thm: lmu property} correspond to the formulas of the generalized $\mu$-function as \eqref{eq: gmu formula 1} $\leftrightarrow$ \eqref{eq: lmu formula 1}, \eqref{eq: gmu formula 2} $\leftrightarrow$ \eqref{eq: lmu formula 2}, \eqref{eq: gmu formula 3} $\leftrightarrow$ \eqref{eq: lmu formula 3}, \eqref{eq: gmu formula 4} $\leftrightarrow$ \eqref{eq: lmu formula 3.5}, \eqref{eq: gmu formula 5} $\leftrightarrow$ \eqref{eq: lmu formula 4}, \eqref{eq: gmu formula 6} $\leftrightarrow$ \eqref{eq: lmu formula 5} and \eqref{eq: gmu formula 7} $\leftrightarrow$ \eqref{eq: lmu formula 6}.

Next, according \eqref{eq: Schur Fibo} and \eqref{eq: corresponding between q-Fibonacci and Raman ujan eq}, we introduce
\begin{align*}
M_n(x;q):=-iq^\frac{1}{8}l\widehat{\mu}(x,q^{n-1}/x)=\frac{1}{(x,q)_\infty\theta(q^n/x)}{}_1\psi_2\hyper{x}{0,0}{q}{x q^{1-n}}. 
\end{align*}
\begin{thm}
\label{thm: M_n and q-Fibonacci}
For $n\in\Zz$, we obtain
\begin{align}
\label{eq: M_n formula 1}
M_n(x;q)
  &=
  M_n(x^{-1},q)=M_n(x q,q), \\
\label{eq: M_n formula 2}  
M_n(x;q)
  &=
  \frac{1}{(q^{n+1}/x,q)_\infty\theta(x)}
  {}_0\psi_2\hyper{-}{xq^{-n},0}{q}{\frac{x^2}{q^n}} \\
\label{eq: M_n formula 3}
  &=
  \frac{1}{(x,q^{1-n}/x)_\infty}\sum_{k\in\Zz}\left(1-x^2q^{2k-n-1}\right)\frac{(x)_k}{(xq^{-n})_k}q^{\frac{k(5k-7)}{2}}(-x^3q^{-n})^k, \\
\label{eq: M_n formula 4}
  {}_0\phi_1\hyper{-}{0}{q}{q^{n+1}}&=\frac{(-1)^nq^{-\frac{n(n-1)}{2}}\theta(x)^2\theta(y)^2}{y\theta(x/y)\theta(xy)}(M_{n-1}(x;q)-M_{n-1}(y;q)), \\
  M_n(x;q)&=G(q)H(q)\left(\frac{\theta(x^5q^2;q^5)}{\theta(x q)\theta(x^2)}+\frac{\theta(x^5q^3;q^5)}{\theta(x)\theta(x^2q)}\right)T_{n-1}(q)
  \nonumber\\
\label{eq: M_n and S_n and T_n}
  &\quad+G(q)H(q)\left(\frac{\theta(x^5q;q^5)}{\theta(x)\theta(x^2)}+\frac{\theta(x^5q^4;q^5)}{\theta(x q)\theta(x^2q)}\right)S_n(q). 
\end{align} 
In particular, when $n=0$ and $1$, we have 
\begin{align}
\label{eq: M_0}
  M_{0}(x;q)
  &=
  \frac{1}{(q/x,q)_\infty\theta(x)}{}_0\psi_2\hyper{-}{x,0}{q}{x^2}
  =
  G(q)H(q)\left(\frac{\theta(x^5q^2;q^5)}{\theta(x q)\theta(x^2)}+\frac{\theta(x^5q^3;q^5)}{\theta(x)\theta(x^2q)}\right), \\
\label{eq: M_1}
M_1(x;q)
  &=
  \frac{1}{(q^2/x,q)_\infty\theta(x)}{}_0\psi_2\hyper{-}{x/q,0}{q}{\frac{x^2}{q}}
  =
  G(q)H(q)\left(\frac{\theta(x^5q;q^5)}{\theta(x)\theta(x^2)}+\frac{\theta(x^5q^4;q^5)}{\theta(x q)\theta(x^2q)}\right). 
\end{align}
\end{thm}

Moreover, we calculate the following Wronski determinants involving the quadratic relations of the Rogers-Ramanujan continued fraction. 
\begin{thm}
\label{thm: lmu Wronskian}
For $m, n\in\Zz$, we have 
 \begin{align}
 &l\widehat{\mu}\left(\frac{x}{c}q^{n-1},yc\right)l\widehat{\mu}\left(xq^{m-1},y\right)
 -l\widehat{\mu}\left(\frac{x}{c}q^{m-1},yc\right)l\widehat{\mu}\left(xq^{n-1},y\right)
   \nonumber \\
 \label{eq: Wronskian relation 1}
 &=\frac{q^{-\frac{1}{4}}\theta(c)\theta(yc/x)}{\theta(c/x)\theta(yc)\theta(x)\theta(y)}\left(S_{m}(xy,q)T_{n-1}(xy,q)-S_n(xy,q)T_{m-1}(xy, q)\right), \\
 &(-xy)^nq^\frac{n(n-1)}{2}{}_0\phi_1\hyper{-}{0}{q}{xyq^{n+1}}l\widehat{\mu}\left(xq^{m-1},y\right)-(-xy)^mq^\frac{m(m-1)}{2}{}_0\phi_1\hyper{-}{0}{q}{xyq^{m+1}}l\widehat{\mu}\left(xq^{n-1},y\right)
   \nonumber \\
 \label{eq: Wronskian relation 2}
 &=iq^{-\frac{1}{8}}\left(S_{m}(xy,q)T_{n-1}(xy,q)-S_n(xy,q)T_{m-1}(xy, q)\right). 
 \end{align}
In particular, when $m=n+1$, we have
 \begin{align}
 \label{eq: n-Wronskian relation 1}
 l\widehat{\mu}\left(\frac{x}{c}q^{n-1},y c\right)l\widehat{\mu}\left(xq^{n},y\right)-l\widehat{\mu}\left(\frac{x}{c}q^{n},yc\right)l\widehat{\mu}\left(xq^{n-1},y\right)
   &=
   \frac{(-xy)^nq^{\frac{n(n-1)}{2}-\frac{1}{4}}\theta(c)\theta(yc/x)}{\theta(c/x)\theta(yc)\theta(x)\theta(y)}, \\
 \label{eq: n-Wronskian relation 2}
{}_0\phi_1\hyper{-}{0}{q}{xyq^{n+1}}l\widehat{\mu}\left(xq^n,y\right)+ xyq^n{}_0\phi_1\hyper{-}{0}{q}{xyq^{n+2}}l\widehat{\mu}\left(xq^{n-1},y\right)&=iq^{-\frac{1}{8}}. 
 \end{align}
Furthermore, when $n=0$, we have
\begin{align}
\label{eq: lmu Wronskian 1}
l\widehat{\mu}\left(\frac{x}{cq},yc\right)l\widehat{\mu}\left(x,y\right)-l\widehat{\mu}\left(\frac{x}{c},yc\right)l\widehat{\mu}\left(\frac{x}{q},y\right)
  &=
  \frac{q^{-\frac{1}{4}}\theta(c)\theta(yc/x)}{\theta(c/x)\theta(yc)\theta(x)\theta(y)}, \\
\label{eq: lmu Wronskian 2}
{}_0\phi_1\hyper{-}{0}{q}{xyq}l\widehat{\mu}(x,y)+xy{}_0\phi_1\hyper{-}{0}{q}{xyq^2}l\widehat{\mu}(x/q,y)
  &=
  iq^{-\frac{1}{8}}. 
\end{align}
\end{thm}

\begin{coro}
\label{thm: lmu identities}
For $m, n\in\Zz$, we have 
 \begin{align}
 &M_n(y,q)M_m(x,q)-M_m(y,q)M_n(x,q)\nonumber \\
 &=x\frac{\theta(xy)\theta(y/x)}{\theta(x)^2\theta(y)^2}(S_m(q)T_{n-1}(q)-S_n(q)T_{m-1}(q))
  , \\
 &(-1)^nq^\frac{n(n-1)}{2}{}_0\phi_1\hyper{-}{0}{q}{q^{n+1}}M_m(x;q)-(-1)^mq^\frac{m(m-1)}{2}{}_0\phi_1\hyper{-}{0}{q}{q^{m+1}}M_n(x;q)
   \nonumber \\
 \label{eq: Wronskian relation}
 &=S_{m}(q)T_{n-1}(q)-S_n(q)T_{m-1}(q). 
 \end{align}
In particular, when $m=n+1$, we have 
 \begin{align}
 M_n(y,q)M_{n+1}(x,q)-M_{n+1}(y,q)M_n(x,q)
   &=
   (-1)^{n}xq^\frac{n(n-1)}{2}\frac{\theta(xy)\theta(y/x)}{\theta(x)^2\theta(y)^2}
   , \\
 \label{eq: n-Wronskian relation}
 {}_0\phi_1\hyper{-}{0}{q}{q^{n+1}}M_{n+1}(x;q)+q^n{}_0\phi_1\hyper{-}{0}{q}{q^{n+2}}M_n(x;q)
   &=1. 
 \end{align}
Furthermore, when $n=0$, we have
\begin{align}
M_0(y,q)M_1(x,q)-M_1(y,q)M_0(x,q)
  &=
  x\frac{\theta(xy)\theta(y/x)}{\theta(x)^2\theta(y)^2}
  ,
  \\
G(q)M_1(x;q)+H(q)M_{0}(x;q)
  &=
  G(q)M_1\left(q^\frac{1}{5};q\right)+H(q)M_{0}\left(q^\frac{1}{5};q\right)
  \nonumber \\
  &=
  \frac{\eta(5\tau)}{\eta(\tau/5)}\left(\frac{1}{R(q)}-1-R(q)\right)
  \nonumber \\
\label{eq: RR and little mu}
  &=
  1, 
\end{align}
where $R(q)$ is the Rogers-Ramanujan continued fraction: 
$$R(q):=q^\frac{1}{5}\frac{H(q)}{G(q)}=\cfrac{q^\frac{1}{5}}{1+\cfrac{q}{1+\cfrac{q^2}{1+\cdots}}}, $$
and $\eta(\tau)$ is the Dedekind eta function: 
\begin{align*}
\eta(\tau):=q^\frac{1}{24}(q)_\infty. 
\end{align*}
\end{coro}


Finally, in Appendix \ref{app: A}, we present fundamental solutions of variations on the Ramanujan equation \eqref{eq: Ramanujan equation} by the little $\mu$-function,
and in Appendix \ref{app: B}, we present explicit formulas of the $q,t$-Fibonacci sequences $S_n(t,q)$ and $T_{n-1}(t,q)$. 

\section{Proofs of the main results}
\begin{proof}[Proof of Theorem $\ref{thm: lmu and gmu}$]
First, from the definitions of the $q$-Borel transformation and the $q$-Laplace transformation \eqref{eq: q-Borel Laplace transformation}, we have
\begin{align}
\label{eq: image of Ramanujan}
\mathcal{L}_q\circ\mathcal{B}_q\left(\widetilde{f}_1\right)(x,\lambda)=\sum_{n\in\Zz}\frac{1}{(q)_\infty\theta(-\lambda q^n/x)}{}_1\phi_0\hyper{0}{-}{q}{-\lambda q^n}. 
\end{align}
From the $q$-binomial theorem \cite[p.8, (1.3.2)]{GR}
\begin{align*}
{}_1\phi_0\hyper{a}{-}{q}{x}=\frac{(ax)_\infty}{(x)_\infty}
\end{align*}
and the relations satisfied by the theta function
\begin{align}
\label{eq: theta relation}
\theta(x)+x\theta(x q)=0,\quad \theta(x)+x\theta(x^{-1})=0, 
\end{align}
the image \eqref{eq: image of Ramanujan} is written as 
\begin{align*}
\mathcal{L}_q\circ\mathcal{B}_q\left(\widetilde{f}_1\right)(xy,-x)&=\frac{1}{(q)_\infty\theta(y q)}\sum_{n\in\Zz}\frac{(-1)^ny^{-n}q^\frac{n(n-1)}{2}}{(x q^n)_\infty}
  =\frac{1}{(x,q)_\infty\theta(y q)}{}_1\psi_2\hyper{x}{0,0}{q}{\frac{1}{y}}=-iq^\frac{1}{8}l\widehat{\mu}(x,y). 
\end{align*}

By the definition of the generalized $\mu$-function \eqref{eq: defi of gmu} and \eqref{eq: gmu formula 2}, we have 
\begin{align*}
  \lim_{a\to0}(xy)^{-\frac{\alpha }{2}}a^{-\frac{1}{2}}\widehat{\mu}(xq,y;a)
  =
  \lim_{a\to0}(xy)^{-\frac{\alpha }{2}}a^{-\frac{1}{2}}\widehat{\mu}(x,qy;a)
  =
  \lim_{a\to0}iq^{-\frac{1}{8}}\frac{(ax)_\infty}{(x,q)_\infty\theta(y q)}
    {}_1\psi_2\hyper{x}{0,a x}{q}{\frac{1}{y}}=l\widehat{\mu}(x,y). 
\end{align*}

\end{proof}
\begin{rmk}
The little $\mu$-function also coincides with the image of $\mathcal{L}_q$ of the following formal solution around $x=\infty$: 
\begin{align*}
g_\infty(x)=\sum_{n=0}^\infty\frac{1}{(q)_n}q^\frac{n(n+1)}{2}\left(\frac{1}{xy}\right)^n={}_0\phi_0\hyper{-}{-}{q}{-\frac{q}{xy}}=\left(-\frac{q}{xy}\right)_\infty
\end{align*}
of the following one order $q$-difference equation: 
\begin{align}
\label{eq: RB equation}
[xyT_x-1-xy]g(x)=0. 
\end{align} 
In fact, a bit calculation shows that
\begin{align*}
\widetilde{f}_\infty(x,\lambda)=\mathcal{L}_q\left(\frac{g_\infty(x)}{\theta(-xy)}\right)=\mathcal{L}_q\left(\frac{1}{(-xy)_\infty}\right)=\mathcal{L}_q\circ\mathcal{B}_q\left(\widetilde{f}_1\right)(x,\lambda). 
\end{align*}

\end{rmk}
\begin{proof}[Proof of Theorem $\ref{thm: lmu property}$]
By using \eqref{eq: lmu and gmu} of Theorem \ref{thm: lmu and gmu}, we have \eqref{eq: lmu formula 1}--\eqref{eq: lmu formula 6} from \eqref{eq: gmu formula 1}--\eqref{eq: gmu formula 7}, respectively. 
From \eqref{eq: lmu formula 1}, $l\widehat{\mu}(x q^{n-1},y)$  and the right hand side of \eqref{eq: lmu formula 7} satisfies the recursion \eqref{eq: q-Fibonacci}. 
Since \eqref{eq: lmu formula 7} holds for $n=0$ and $n=1$, we obtain \eqref{eq: lmu formula 7}. 
\end{proof}

\begin{rmk}
Putting $c=q$ or $c=x/y$ in \eqref{eq: lmu formula 6}, since $\theta(q)=\theta(1)=0$, we also obtain \eqref{eq: lmu formula 2}. 
\end{rmk}

\begin{proof}[Proof of Theorem $\ref{thm: M_n and q-Fibonacci}$]
The formulas \eqref{eq: M_n formula 1}--\eqref{eq: M_n formula 4} follow from \eqref{eq: lmu formula 2}, \eqref{eq: lmu formula 3}, \eqref{eq: lmu formula 3.5} and \eqref{eq: lmu formula 6} of Theorem \ref{thm: lmu property} immediately. 
Moreover, the equality \eqref{eq: M_0} and \eqref{eq: M_1} are obtained by putting $n=0$ and $n=1$ in \eqref{eq: M_n formula 3}, respectively. 
The formula \eqref{eq: M_n and S_n and T_n} follows from \eqref{eq: lmu formula 7}, \eqref{eq: M_0} and \eqref{eq: M_1}. 
\end{proof}

\begin{proof}[Proof of Theorem $\ref{thm: lmu Wronskian}$]
First, we show \eqref{eq: lmu Wronskian 2}. 
We set the left hand side of \eqref{eq: lmu Wronskian 2} as
\begin{align*}
\mathcal{M}_0(x,y):={}_0\phi_1\hyper{-}{0}{q}{xyq}l\widehat{\mu}(x,y)+xy{}_0\phi_1\hyper{-}{0}{q}{xyq^2}l\widehat{\mu}(x/q,y). 
\end{align*}
By the definition of the little $\mu$-function, $\mathcal{M}_0(x,y)$ has simple poles at $x, y\in q^\Zz$. 
The Wronskian of the fundamental solution $\widetilde{f}_0(xy), l\widehat{\mu}(x,y)$ of the Ramanujan equation \eqref{eq: Ramanujan equation}
\begin{align*}
\mathcal{W}(x,y)=\det \begin{bmatrix}
\widetilde{f}_0(xy/q) & \widetilde{f}_0(x y) \\
l\widehat{\mu}(x/q,y) & l\widehat{\mu}(x,y)
\end{bmatrix}
  &=\frac{1}{\theta(xy)}\mathcal{M}_0(x,y)
\end{align*}
satisfies the following pseudo-periodicities: 
\begin{align*}
\mathcal{W}(x q,y)=\mathcal{W}(x,y q)=-x y \mathcal{W}(x,y), 
\end{align*}
from \eqref{eq: lmu formula 1} and \eqref{eq: lmu formula 2}. 
Thus the function $\mathcal{M}_0(x,y)$ is a pseudo-constant with respect to $x$ and $y$. 
From Liouville's theorem, $\mathcal{M}_0(x,y)$ is a constant function depending only on the variable $q$. 
Hence, we have
\begin{align}
\label{eq: RR and little mu 1}
\mathcal{M}_0(x,y)=\mathcal{M}_0\left(q^\frac{1}{5},q^{-\frac{1}{5}}\right)=iq^{-\frac{1}{8}}\left(G(q)M_1\left(q^\frac{1}{5};q\right)+H(q)M_{0}\left(q^\frac{1}{5};q\right)\right). 
\end{align}
Then, to prove \eqref{eq: lmu Wronskian 2}, we show \eqref{eq: RR and little mu}. 

By putting $x=q^\frac{1}{5}$ in \eqref{eq: M_0} and \eqref{eq: M_1}, we obtain
\begin{align*}
M_0\left(q^\frac{1}{5};q\right)
  &=
  G(q)H(q)\left(\frac{\theta\left(q^3;q^5\right)}{\theta\left(q^\frac{6}{5}\right)\theta\left(q^\frac{2}{5}\right)}+\frac{\theta\left(q^8;q^5\right)}{\theta\left(q^\frac{1}{5}\right)\theta\left(q^\frac{7}{5}\right)}\right)
  =
  -\frac{(q)_\infty}{(q^\frac{1}{5};q^\frac{1}{5})_\infty}\left(q^\frac{1}{5}G(q)+q^\frac{2}{5}H(q)\right), 
  \\
  M_1\left(q^\frac{1}{5};q\right)
  &=
  G(q)H(q)\left(\frac{\theta\left(q^2;q^5\right)}{\theta\left(q^\frac{1}{5}\right)\theta\left(q^\frac{2}{5}\right)}+\frac{\theta\left(q^5;q^5\right)}{\theta\left(q^\frac{6}{5}\right)\theta\left(q^\frac{7}{5}\right)}\right)
  =
  \frac{(q)_\infty}{(q^\frac{1}{5};q^\frac{1}{5})_\infty}G(q). 
\end{align*}
Thus, the function $\mathcal{M}_0(x,y)$ is rewritten as
\begin{align*}
  \mathcal{M}_0(x,y)
  =
  \frac{iq^{-\frac{1}{8}}(q)_\infty}{(q^\frac{1}{5};q^\frac{1}{5})_\infty}\left(G(q)^2-q^\frac{1}{5}G(q)H(q)-q^\frac{2}{5}H(q)^2\right)
  =
  iq^{-\frac{1}{8}}\frac{\eta(5\tau)}{\eta(\tau/5)}\left(\frac{1}{R(q)}-1-R(q)\right).
\end{align*}
The third equality in \eqref{eq: lmu Wronskian 2} follows from \cite[p.48]{W}.

Next, we show \eqref{eq: lmu Wronskian 1}. we set the left hand side of \eqref{eq: lmu Wronskian 1} as
\begin{align*}
\mathcal{M}_1(x,y;c):=\frac{\theta(c/x)\theta(yc)\theta(x)\theta(y)}{\theta(c)\theta(yc/x)}\left(l\widehat{\mu}\left(\frac{x}{cq},yc\right)l\widehat{\mu}\left(x,y\right)-l\widehat{\mu}\left(\frac{x}{c},yc\right)l\widehat{\mu}\left(\frac{x}{q},y\right)\right). 
\end{align*}
Since the function $\mathcal{M}_1(x,y;c)$ satisfies 
$$\mathcal{M}_1(x,y,c q)=\mathcal{M}_1(x,y,c)$$
from \eqref{eq: lmu formula 2}, 
it is a pseudo-constant with respect to the variable $c$ and has simple poles at $c\in (x/y)q^\Zz$. 
From \eqref{eq: lmu formula 4}, we obtain
\begin{align*}
\mathcal{M}_1(x,y,c)=\lim_{c\to 1/y}\mathcal{M}_1(x,y,c)
  &=-iq^{-\frac{1}{8}}\mathcal{M}_0(x,y)=q^{-\frac{1}{4}}. 
\end{align*}

Finally, we show \eqref{eq: Wronskian relation 1}, \eqref{eq: Wronskian relation 2}, \eqref{eq: n-Wronskian relation 1} and \eqref{eq: n-Wronskian relation 2}. 
From \eqref{eq: lmu formula 6}, the functions $l\widehat{\mu}\left(x,y\right)$ and $l\widehat{\mu}\left(\frac{x}{q},y\right)$ are expressed as linear combinations of $l\widehat{\mu}(x q^{m-1},y)$ and $l\widehat{\mu}(x q^{n-1},y)$: 
\begin{align}
l\widehat{\mu}\left(x,y\right)
  &=
  \frac{l\widehat{\mu}(x q^{m-1},y)T_{n-1}(xy,q)-l\widehat{\mu}(x q^{n-1},y)T_{m-1}(xy,q)}{S_m(xy,q)T_{n-1}(xy,q)-S_n(xy,q)T_{m-1}(xy,q)}
  , \nonumber \\
\label{eq: lmu hyouji}
  l\widehat{\mu}\left(\frac{x}{q},y\right)
  &=
  \frac{l\widehat{\mu}(x q^{n-1},y)S_m(xy,q)-l\widehat{\mu}(x q^{m-1},y)S_n(xy,q)}{S_m(xy,q)T_{n-1}(xy,q)-S_n(xy,q)T_{m-1}(xy,q)}. 
\end{align}
By substituting the expressions \eqref{eq: lmu hyouji} in \eqref{eq: lmu Wronskian 1}, we obtain \eqref{eq: Wronskian relation 1}. 
Moreover, taking the limit $c\to 1/y$ in \eqref{eq: Wronskian relation 1}, we get \eqref{eq: Wronskian relation 2} from \eqref{eq: lmu formula 4}. 

From \eqref{eq: q-Fibonacci}, the determinant
\begin{align*}
\mathcal{F}_n(xy):=\det \begin{bmatrix}
S_{n+1}(xy,q) & S_{n}(xy,q) \\
T_{n}(xy,q) & T_{n-1}(xy,q)
\end{bmatrix}=S_{n+1}(xy,q)T_{n-1}(xy,q)-S_{n}(xy,q)T_{n}(xy,q)
\end{align*}
satisfies the following recursions
\begin{align*}
\mathcal{F}_n(xy)=-xyq^{n-1}\mathcal{F}_{n-1}(xy)=\cdots=(-xy)^nq^\frac{n(n-1)}{2}. 
\end{align*}
Then, if $m=n+1$ in \eqref{eq: Wronskian relation 1} and \eqref{eq: Wronskian relation 2},  we have \eqref{eq: n-Wronskian relation 1} and \eqref{eq: n-Wronskian relation 2}. 
\end{proof}

\begin{appendices}
\section{Variations on the Ramanujan Equation}
\label{app: A}

The following convex hull related to a linear $q$-difference equation
\begin{align*}
[a_n(x)T_x^n+a_{n-1}(x)T_x^{n-1}+\cdots +a_0(x)]f(x)=0, \quad a_l(x)=\sum_{k\geq0} c_{k,l}x^k \in\Cz[x], 
\end{align*}
is called the Newton-Puiseux diagram: 
\begin{align*}
\{(k,l)\in\Rz^2;\ c_{k,l}\neq0\}. 
\end{align*}
Putting $x\mapsto x/y, f(x/y)=v(x)$ in \eqref{eq: Ramanujan equation}, we rewrite
\begin{align}
\label{eq: Ramanujan equation 0}
[T_x^2-T_x-x q]v(x)=0. 
\end{align}
Then, the functions $\widetilde{f}_0(x), l\widehat{\mu}(x/y,y)$ are fundamental solutions of \eqref{eq: Ramanujan equation 0}. 
In particular, since 
\begin{align*}
\lim_{y\to 1}\theta(y) l\widehat{\mu}(x/y,y)=iq^{-\frac{1}{8}}\widetilde{f}_0(x), 
\end{align*}
we see that fundamental solutions of \eqref{eq: Ramanujan equation 0} are expressed solely in terms of the little $\mu$-function. 

On the other hand, the most degenerate second order linear $q$-difference equations of Laplace type \eqref{eq: Laplace} excluding those of constant coefficients are the following six type:   
\begin{align}
  [aT_x^2-T_x-bxq]v_1(x)&=0, 
  \quad
  [axqT_x^2-T_x-b]v_2(x)=0, 
  \quad
  [axqT_x^2-T_x-bx q]v_3(x)=0, 
  \nonumber \\
  \label{eq: various Ramanujan equation} 
  [aT_x^2-xqT_x-bxq]v_4(x)&=0, 
  \quad
  [ax qT_x^2-x qT_x-b]v_5(x)=0, 
  \quad
  [aT_x^2-xqT_x-b]v_6(x)=0
\end{align}
The each Newton-Puiseux diagram of the $q$-difference equations \eqref{eq: various Ramanujan equation} is as follows: 
\begin{figure}[htbp]
\centering
\begin{minipage}[b]{0.3\columnwidth}
    \centering
    \includegraphics[width=0.7\columnwidth]{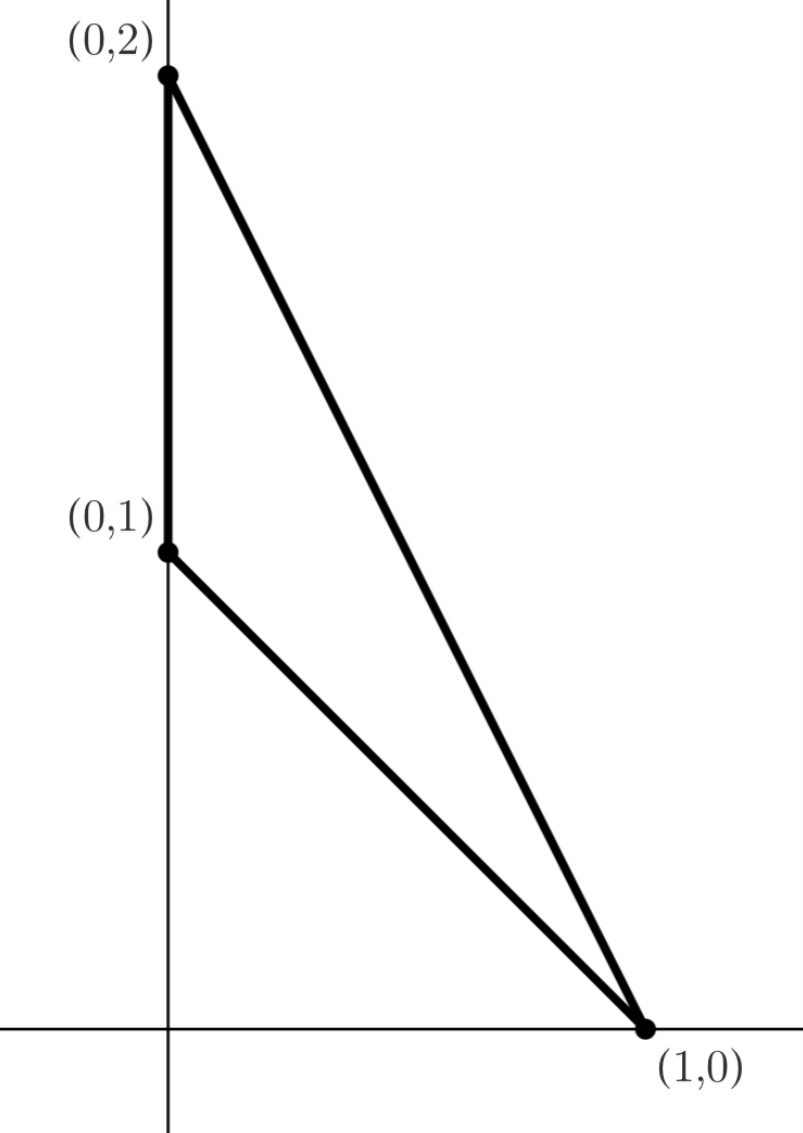}
    \caption{$v_1$}
    \label{fig:u1}
\end{minipage}
\begin{minipage}[b]{0.3\columnwidth}
    \centering
    \includegraphics[width=0.7\columnwidth]{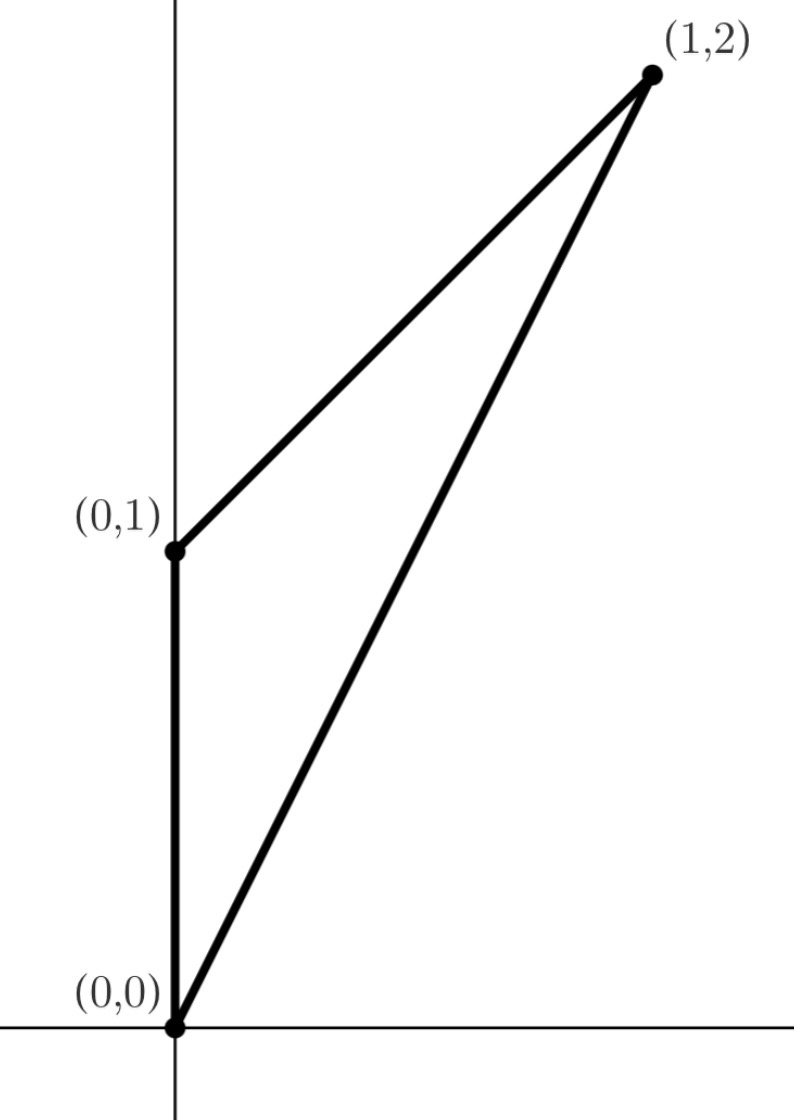}
    \caption{$v_2$}
    \label{fig:u2}
\end{minipage}
\begin{minipage}[b]{0.3\columnwidth}
    \centering
    \includegraphics[width=0.7\columnwidth]{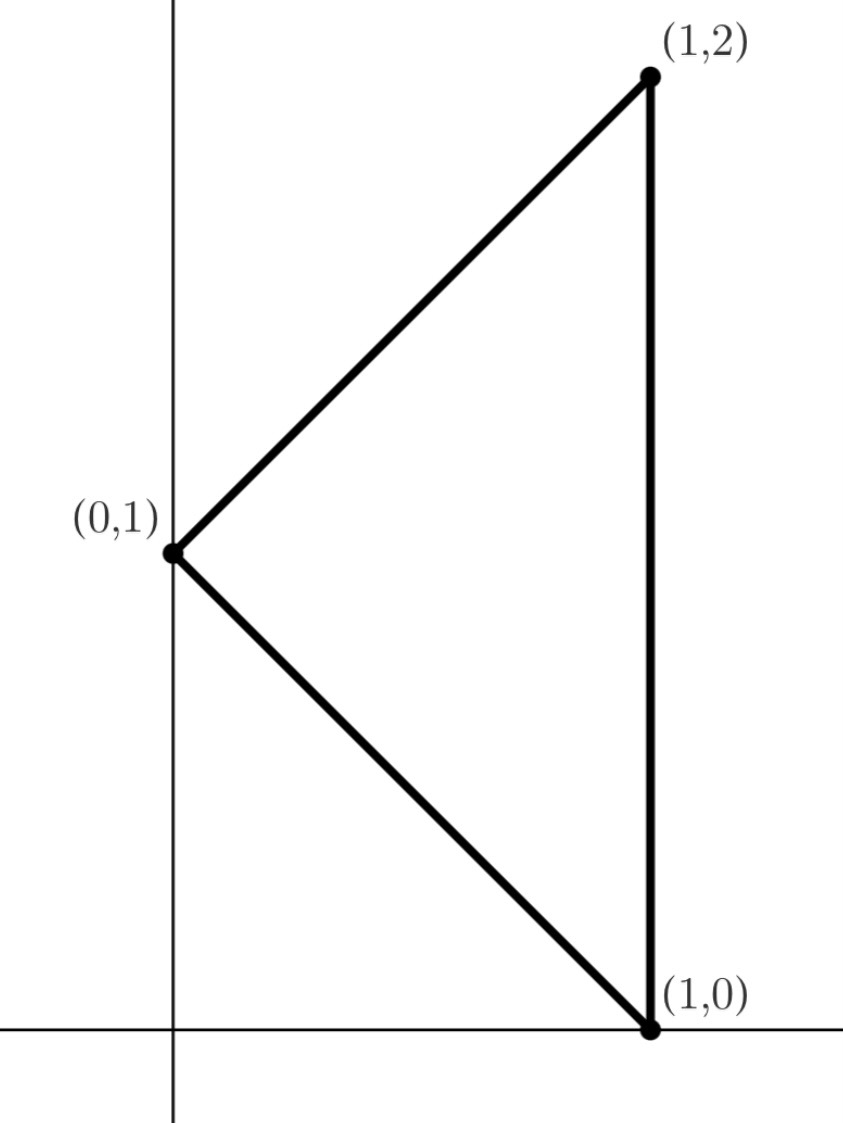}
    \caption{$v_3$}
    \label{fig:u3}
\end{minipage}
\end{figure}

\begin{figure}[htbp]
\centering
\begin{minipage}[b]{0.3\columnwidth}
    \centering
    \includegraphics[width=0.7\columnwidth]{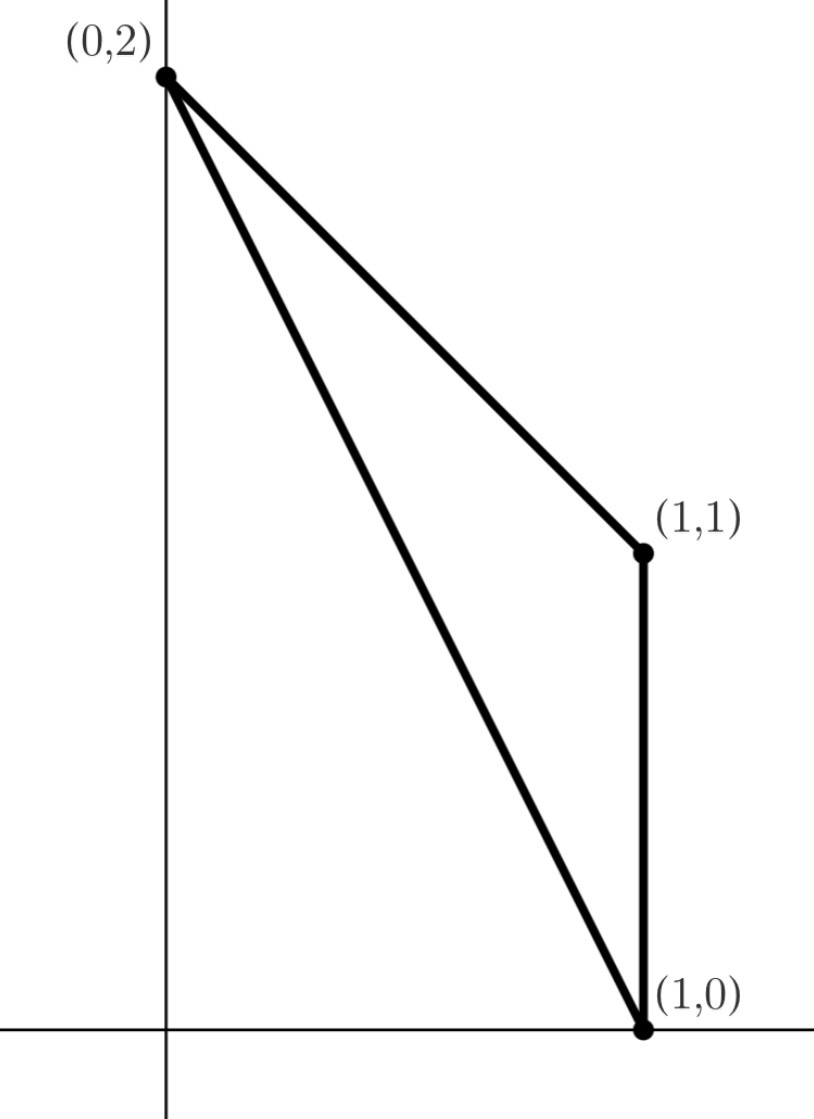}
    \caption{$v_4$}
    \label{fig:u4}
\end{minipage}
\begin{minipage}[b]{0.3\columnwidth}
    \centering
    \includegraphics[width=0.7\columnwidth]{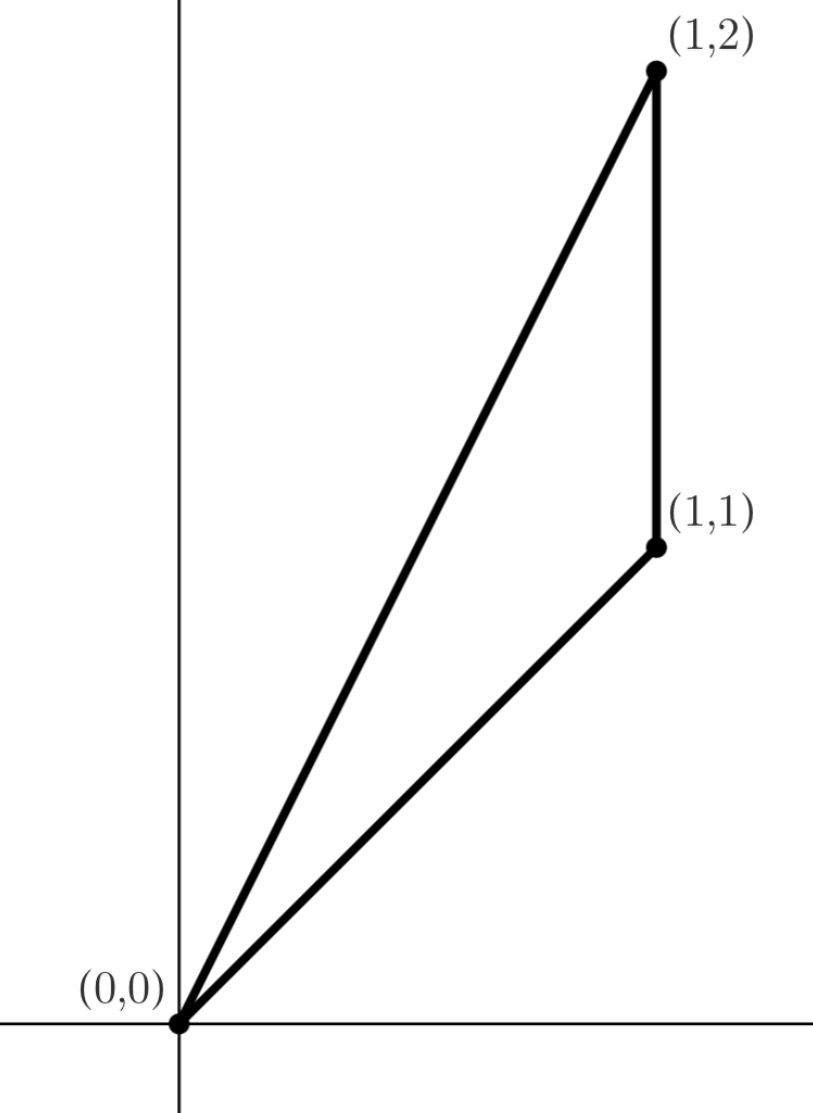}
    \caption{$v_5$}
    \label{fig:u5}
\end{minipage}
\begin{minipage}[b]{0.3\columnwidth}
    \centering
    \includegraphics[width=0.7\columnwidth]{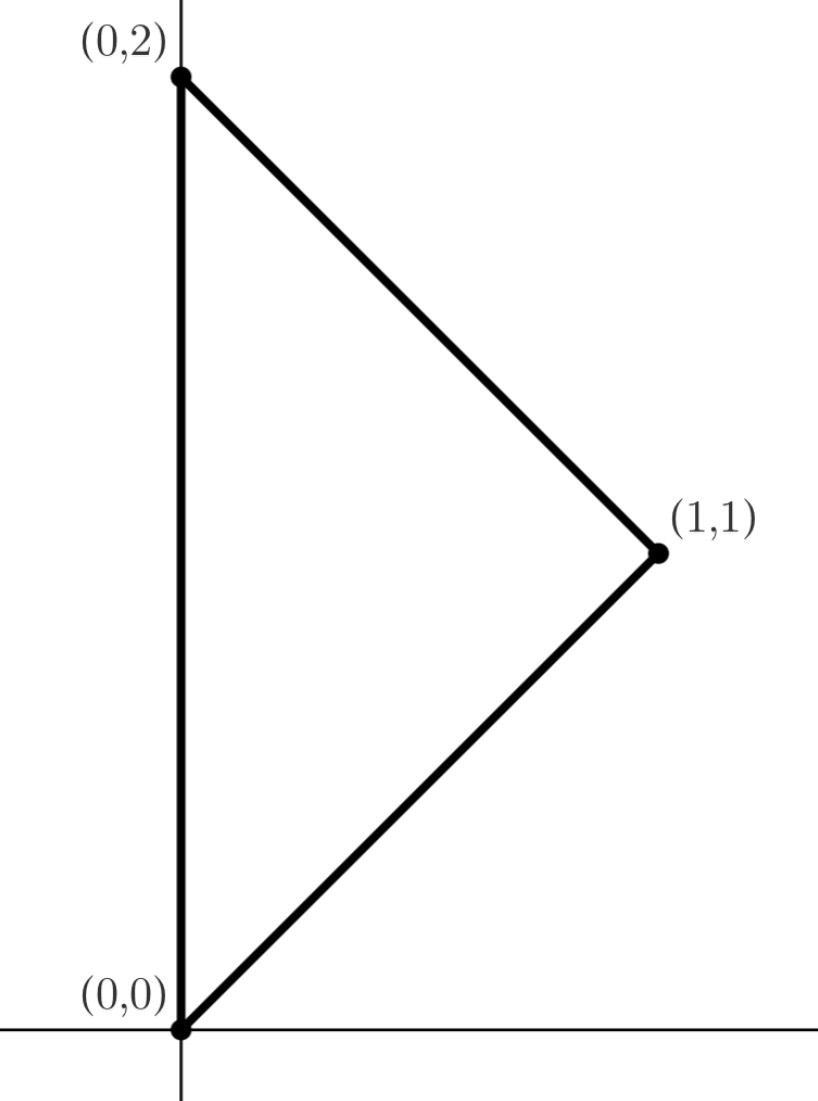}
    \caption{$v_6$}
    \label{fig:u6}
\end{minipage}
\end{figure}
\newpage
Figure \ref{fig:u1}, \ref{fig:u2}, \ref{fig:u4} and \ref{fig:u5} correspond to the Newton-Puiseux diagrams of $q$-difference equations associated with $\widetilde{f}_0$, while Figure \ref{fig:u3} and \ref{fig:u6} correspond to those of $q$-difference equations associated with the $q$-Airy function $${}_1\phi_1\hyper{0}{-q}{q}{x}=\sum_{n\geq0}\frac{(-1)^nq^\frac{n(n-1)}{2}}{(q^2;q^2)_n}x^n. $$
Moreover, various recursions of $q$-Fibonacci sequences  are also included in \eqref{eq: various Ramanujan equation}. 

Then, assuming $v_0(x)$ as a solution of \eqref{eq: Ramanujan equation 0}, the functions 
\begin{align}
  v_1(x)&=\frac{\theta(-ax)}{\theta(-x)}v_0(abx,q)
  ,\quad 
  v_2(x)=\theta(-ax)v_0(abx/q,q), 
  \quad 
  v_3(x)=\theta(-ax)v_0(ab x^2/q,q^2), 
  \nonumber \\
  \label{eq: various little mu function}
  v_4(x)&=\frac{\theta(x/b)}{\theta(-x)}v_0(ab/x,q), 
  \quad 
  v_5(x)=\theta(xq/b)v_0(ab/xq,q), 
  \quad 
  v_6(x)=\theta(xq/b)v_0(ab/x^2q;q^2)
\end{align}
are solutions of \eqref{eq: various Ramanujan equation} and the little $\mu$-functions give a fundamental solutions under the transformations \eqref{eq: various little mu function}, respectively. 
\begin{exg}
Ismail-Zhang \cite{IZ} introduced the following series as a bilateral version of the Rogers-Ramanujan series: 
\begin{align*}
u_m(a,q):=\sum_{n\in\Zz}\frac{q^{n^2+mn}}{(aq)_n}. 
\end{align*}
This series $u_m(a,q)$ satisfies the following $1$-parameter deformed $q$-Fibonacci recursion: 
\begin{align}
\label{eq: a-Fibonacci}
q^{m+1}u_{m+2}(a,q)+a u_{m+1}(a,q)-u_m(a,q)=0. 
\end{align}
Here, we define 
\begin{align*}
u(a,x,q):=\sum_{n\in\Zz}\frac{q^{n^2}}{(aq)_n}x^n, 
\end{align*}
then, we have $u(a,q^m,q)=u_m(a,q)$ and $u(a,x,q)$ satisfies the following $q$-difference equation: 
\begin{align}
\label{eq: satisfied by u}
[xq T_x^2+a T_x-1]u(a,x,q)=0. 
\end{align}
This equation \eqref{eq: satisfied by u} coincides with the $q$-difference equation satisfied by $\theta(x/a)l\widehat{\mu}(x/ayq,y/a)$. 
More precisely, we have 
\begin{align*}
iq^{-\frac{1}{8}}u(a,x,q)=(1/a,q)_\infty\theta(x/a)l\widehat{\mu}(x/a,1/aq) 
\end{align*}
and \cite[Theorem 3.3]{IZ} from \eqref{eq: lmu formula 3} and \eqref{eq: lmu formula 7}, respectively.
\end{exg}
\end{appendices}

\begin{appendices}
\section{Explicit formulas of $q,t$-Fibonacci sequences}
\label{app: B}
\begin{defi}[\cite{A}, \cite{C}]
For $n\in\Zz$, we define
$$
\mathcal{S}_{n}(t,q)
   :=
   \begin{cases}
   \sum_{j = 0}^{\lfloor \frac{n - 1}{2}\rfloor}
      q^{j^{2}}t^{j}\binom{n - 1 - j}{j}_{q} & (n \not= 0) \\
   0 & (n = 0)
   \end{cases}, 
$$
where
\begin{align*}
\binom{\alpha}{\beta}_q:=\frac{(q)_\alpha}{(q)_\beta(q)_{\alpha-\beta}}. 
\end{align*}
\end{defi}

\begin{prop}
\label{prop: q-Fibonacci}
For $n\in\Zz_{\geq0}$, the following equalities follow:  
\begin{align}
\label{eq: S=S and T=S}
S_n(t,q)=\mathcal{S}_{n}(t,q),\quad T_n(t,q)= t\mathcal{S}_{n}(qt,q). 
\end{align}
\end{prop}
\begin{proof}
From \cite{A} and \cite{C}, since 
\begin{align}
\label{eq: mathS recursion}
\mathcal{S}_n(t,q)&=\mathcal{S}_{n-1}(t,q)+t q^{n-2}\mathcal{S}_{n-2}(t,q), \\
\mathcal{S}_0(t,q)&=0,\quad \mathcal{S}_1(t,q)=1, 
\end{align}
the first equality of \eqref{eq: S=S and T=S} holds. 

On the other hand, For \eqref{eq: mathS recursion}, replacing $n\mapsto n-1, t\mapsto qt$, we have 
\begin{align*}
\mathcal{S}_{n-1}(qt,q)=\mathcal{S}_{n-2}(qt,q)+t q^{n-2}\mathcal{S}_{n-3}(qt,q), 
\end{align*}
Therefore, the sequence $t\mathcal{S}_{n-1}(qt,q)$ satisfies \eqref{eq: q-Fibonacci} with the initial values $t\mathcal{S}_{-1}(qt,q)=\mathcal{S}_{1}(qt,q)-\mathcal{S}_{0}(qt,q)=1$ and $t\mathcal{S}_{0}(qt,q)=0$, the second equalities of \eqref{eq: S=S and T=S} holds. 
\end{proof}
\begin{lem}
\label{lem: q-Fibonacci}
$(1)$\ If a sequence $F_n(t,q)$ satisfies \eqref{eq: q-Fibonacci}, then the sequence $(-t)^nq^\frac{n(n-1)}{2}F_{-n+1}\left(t,q^{-1}\right)$ also satisfies \eqref{eq: q-Fibonacci}. \\
$(2)$\ For all integers $n$, we have: 
\begin{align*}
S_n(t,q)=(-t)^{n-1}q^\frac{n(n-1)}{2}T_{-n}\left(t,q^{-1}\right),\quad T_{n-1}(t,q)=(-t)^{n}q^\frac{n(n-1)}{2}S_{-n+1}\left(t,q^{-1}\right). 
\end{align*}
\end{lem}
\begin{proof}
(1)\ It is clear. \\
(2)\ 
Since the sequences $S_n(t,q)$ and $T_{n-1}(t,q)$ satisfy \eqref{eq: q-Fibonacci}, the sequences 
\begin{align*}
S_n'(t,q):=(-t)^{n}q^\frac{n(n-1)}{2}S_{-n+1}\left(t,q^{-1}\right),\quad T_{n-1}'(t,q):=(-t)^{n-1}q^\frac{n(n-1)}{2}T_{-n}\left(t,q^{-1}\right)
\end{align*}
also satisfy \eqref{eq: q-Fibonacci} from $(1)$ of this Lemma. 
Moreover, since 
\begin{align*}
S_0'(t,q)=1=T_{-1}(t,q), \quad S_1'(t,q)=0=T_0(t,q),\quad T_{-1}'(t,q)=0=S_0(t,q), \quad T_{0}'(t,q)=1=S_1(t,q)
\end{align*}
we have
\begin{align*}
S_n'(t,q)=T_{n-1}(t,q),\quad T_{n-1}'(t,q)=S_n(t,q). 
\end{align*}
\end{proof}

From the above Proposition and Lemma, we immediately obtain the following Theorem and Corollary. 
\begin{thm}
For $n\in\Zz$, we have
\begin{align*}
S_n(t,q)&=
\begin{cases}
\mathcal{S}_n(t,q) & n\geq0\\
-(-t)^{n}q^\frac{n(n-1)}{2}\mathcal{S}_{-n}\left(q^{-1} t,q^{-1}\right) & n<0
\end{cases}, \\
T_n(t,q)&=
\begin{cases}
t\mathcal{S}_n(qt,q) & n \geq0\\
(-t)^{n+1}q^\frac{n(n+1)}{2}\mathcal{S}_{-n}\left(t,q^{-1}\right) & n<0
\end{cases}. 
\end{align*}
\end{thm}

\begin{coro}
For $n\in\Zz$, we have
\begin{align}
\label{eq: Sn}
S_n(q)&=\begin{cases}
\sum_{0\leq 2j\leq n-1}q^{j^2}\binom{n-j-1}{j}_q & n>0\\
(-1)^{n-1}q^\frac{n(n-1)}{2}\sum_{0\leq 2j \leq -n-1}q^{j^2+jn}\binom{-n-j-1}{j}_q & n< 0\\
0 & n=0
\end{cases}, \\
\label{eq: Tn}
T_n(q)&=\begin{cases}
\sum_{0\leq 2j\leq n-1}q^{j^2+j}\binom{n-j-1}{j}_q & n>0\\
(-1)^{n-1}q^\frac{n(n+1)}{2}\sum_{0\leq 2j\leq -n-1}q^{j^2+jn+j}\binom{-n-j-1}{j}_q & n<0\\
0 & n=0
\end{cases}. 
\end{align}
\end{coro}
\end{appendices}

\medskip
\begin{flushleft}
Genki Shibukawa\\ 
Kitami Institute of Technology\\
165, Koen-cho, Kitami, Hokkaido 090-8507\\
Japan\\
E-mail: g-shibukawa@mail.kitami-it.ac.jp\\
\end{flushleft}

\medskip 
\begin{flushleft}
Satoshi Tsuchimi \\
Kindai University \\
3-4-1, Kowakae, Higasi-Osaka, Osaka, 577-8502\\
Japan \\
E-mail: tsuchimi@math.kindai.ac.jp
\end{flushleft}
\end{document}